\documentclass[leqno,12pt,a4paper,twoside]{article}%
\usepackage{amssymb}
\usepackage{amsfonts}
\usepackage{amsmath}
\usepackage{authblk}
\usepackage{graphicx}
\usepackage{ifthen}
\usepackage{comment}
\usepackage{xcolor}
\usepackage{enumitem}
\usepackage{lipsum,multicol}
\usepackage[T1]{fontenc}
\usepackage[english]{babel}
\usepackage[english]{babel}
\setlist[itemize]{noitemsep, topsep=0pt}
\setlist[enumerate]{noitemsep, topsep=0pt}
\setlist[itemize]{leftmargin=*}
\setlist[enumerate]{leftmargin=*}
\providecommand{\U}[1]{\protect\rule{.1in}{.1in}}
\providecommand{\norm}[1]{\left\lVert#1\right\rVert}%//.//

\providecommand{\pr}[1]{\left(#1\right)} %(.)
\newcommand{\normi}[3]{\norm{#1}_{
		\ifthenelse{\equal{#2}{1}}{H_0^1\pr{\mathcal{O}_{#3}}}{%
			\ifthenelse{\equal{#2}{-1}}{H^{-1}\pr{\mathcal{O}_{#3}}}{}}}}

\makeatletter
\newcommand{\subjclass}[2][2020]{%
	\let\@oldtitle\@title%
	\gdef\@title{\@oldtitle\footnotetext{\textbf{#1 \emph{Mathematics subject classification.}} #2}}%
}
\newcommand{\keywords}[1]{%
	\let\@@oldtitle\@title%
	\gdef\@title{\@@oldtitle\footnotetext{\textbf{\emph{Key words.}} #1.}}%
}
\makeatother

\newtheorem{theorem}{Theorem}

\newtheorem{definition}[theorem]{Definition}

\newtheorem{lemma}[theorem]{Lemma}

\newenvironment{proof}[1][Proof]{\noindent\textbf{#1.} }{\ \rule{0.5em}{0.5em}}

\author[1,2,*]{Ionu\c t Munteanu}

\affil[1]{Faculty of Mathematics, Al. I. Cuza University, Bd. Carol I, 11, Iasi 700506, Romania}
\affil[2]{O. Mayer Institute of Mathematics, Romanian Academy, Bd. Carol I, 8, Iasi 700505, Romania}
\affil[*]{e-mail: ionut.munteanu@uaic.ro}
\title{A simple pole-shifting gain matrix $K$  which avoids solving Lyapunov equations}

\date{}

\begin{document}
	\maketitle
	\begin{abstract}
		It is well known that if $A\in\mathbb{C}^{N\times N}$ and $B\in\mathbb{C}^{N\times M}$ form a controllable pair (in the sense that the Kalman matrix $[B\ |\ AB\ | \ \dots\ |\ A^{N-1}B]$ has full rank) then, there exists $K\in\mathbb{C}^{M\times N}$ such that the matrix $A+BK$ has only eigenvalues with negative real parts. The matrix $K$ is not unique, and is usually defined by  a solution of a Lyapunov equation, which, in case of large $N$, is not easily manageable from the computational point of view. In this work, we show that, for general matrices $A$ and $B$, if they satisfy the controllability Kalman rank condition, then $$K=-\overline{B}^\top\sum_{k=1}^N\left[(\overline{A}^\top+\gamma_kI)^{-1}\right]\left\{\sum_{k=1}^N\left[(A+\gamma_kI)^{-1}B\overline{B}^\top(\overline{A}^\top+\gamma_kI)^{-1}\right]\right\}^{-1}$$ ensures that the matrix $A+BK$ has all the eigenvalues with the real part less than $-\gamma_1$. Here, $0<\gamma_1<\gamma_2<\dots<\gamma_N$ are $N$ positive numbers, large enough such that $A+\gamma_kI$ is invertible, for each $k$.
		
	\end{abstract}


	
	\section{Introduction}
	Let $A\in\mathbb{C}^{N\times N}$ and $B\in\mathbb{C}^{N\times M}$. In this work, we will discuss on the linear time invariant controlled system
	\begin{equation}\label{e1}\dot{x}(t)=Ax(t)+Bu(t),\ t>0;\ x(0)=x_0.\end{equation}In case the uncontrolled $(u=0)$ homogeneous system
	$$\dot{x}(t)=Ax(t)$$ fails to be asymptotically stable, one of the tasks of the control analyst is to use the control $u$ in such a way as to remedy this situation.  Because of simplicity for both implementation and analysis, the
	traditionally favored means for accomplishing this objective is the use of a linear feedback relation 
	$$u(t)=Kx(t),\ t\geq 0$$where the control $u(\cdot)$ is determined as a linear function of the current state $x(\cdot)$. The problem now becomes that of choosing the feedback matrix $K\in\mathbb{C}^{M\times N}$, in such a way that once plugged the feedback law $u=Kx$ into \eqref{e1} it ensures that the solution $x(\cdot)$ to the closed-loop system
	$$\dot{x}(t)=(A+BK)x(t),\ t>0;\ x(0)=x_0;$$ satisfies
	$$\|x(t)\|_N\leq \mathcal{C}e^{-\mu t}\|x_0\|_N,\ \forall t\geq 0,$$where $\mathcal{C},\mu$ are positive constants, and $\|\cdot\|_N$ stands for the euclidean norm in $\mathbb{C}^N.$ In such case, we say that this $u$ asymptotically exponentially stabilizes \eqref{e1}.
	
	The feedback matrix $K$ is not unique. Usually, is solution to a Lyapunov equation which, for higher dimension, is complicated to determine (see    below). Moreover, $A$ and $B$ may have complex entries. This makes even harder to solve the Lyapunov equation.  For details, see e.g. \cite{1}.  In this work, we provide an explicit direct form of 
	such matrix $K$ (see Theorem \ref{t} below), easily manageable from the computational point of view. We emphasize that the problem of finding $K$ for a controllable pair $\left\{A,B\right\}$ is the core of the control theory. Therefore, more or less each result regarding controllability or stabilizability of differential systems rely on such problem.
	
	\section{Theory and the main result}
	Let $\left<\cdot,\cdot\right>_N$ stand for the euclidean scalar product in $\mathbb{C}^N$. We say that a matrix $P\in\mathbb{C}^{N\times N}$ is a Hermitian matrix if $P=\overline{P}^\top.$ Here, for $z\in\mathbb{C}$ we set $\overline{z}$ for its complex conjugate; and for a matrix $P$ we set $P^\top$ for its transpose. We say that a Hermitian matrix $P$ is positive definite (positive semi-definite) if $\left<Pz,z\right>_N>0$ for all $z\in\mathbb{C}^N\setminus\left\{0\right\}$ ($\left<Pz,z\right>_N\geq 0,\ \forall z\in\mathbb{C}^N$, respectively). Below, $I$ stands for the identity matrix, while $O$ stands for the matrix with all entries equal to zero.
	\begin{definition}
		We say that the system \eqref{e1} is controllable in time $T>0$ if for any $x_0,x_1\in\mathbb{C}^N$, there exists  $u\in L^2(0,T;\ \mathbb{C}^N)$ such that the corresponding solution of \eqref{e1} satisfying $x(0)=x_0$ also satisfies $x(T)=x_1$.
		\end{definition}
	The controllability of \eqref{e1} implies the asymptotic exponential stabilizability of \eqref{e1}.
	It is well known that the controllability of \eqref{e1} is equivalent to the Kalman rank condition, i.e., the matrix
	$$[B\ |\ AB\ |\ A^2B\ |\ \dots |\ A^{N-1}B]\in \mathbb{C}^{N\times NM}$$ has full rank. (For details on control problems associated to linear differential systems one can check, e.g., \cite{1}.)  This implies that the Kalman controllability matrix associated to the pair $\left\{-\rho I-A,B\right\}$ is also of full rank, for all $\rho\in\mathbb{C}$. Let $\rho>0$ be sufficiently large such that $-\rho I-A$ has all the eigenvalues with negative real part. The above implies that 
	the controllability Gramian matrix
	$$W:=\int_0^\infty e^{(-\rho I-A)t}B\overline{B}^\top e^{(-\rho I-\overline{A}^\top) t}dt $$ is positive definite, in particular is invertible. In this case, the linear feedback control law
	$$u=\frac{1}{2}\overline{B}^\top W^{-1}x=Kx$$asymptotically exponentially stabilizes \eqref{e1}. It is easy to see that $W$ satisfies the so-called Lyapunov equation
	$$(-\rho I-A)W+W(-\rho I-\overline{A}^\top)=-B\overline{B}^\top.$$ In order to get the form of $W$ (and implicitly, the form of $K$) one must deal with this Lyapunov equation. This can be solved by matrix factorization methods, in particular the Bartels-Stewart algorithm can be used, see \cite{2}. However, the computational cost of such algorithm, $N^3$ flops, can be prohibitive for large $N$. Therefore, the controllability Gramian matrix $W$, which defines the matrix $K$, although  is expressed in a simple explicit form, from the computational point of view it might be quite challenging.
	
	In this work, we want to design a better (from the computational point of view) matrix $K$. More exactly, we want to show that
	\begin{theorem}\label{t}
		Assume that the system \eqref{e1} is controllable. Then, the control $u=Kx,$ where
		$$K=-\overline{B}^\top\sum_{k=1}^N\left[(\overline{A}^\top+\gamma_kI)^{-1}\right]\left\{\sum_{k=1}^N\left[(A+\gamma_kI)^{-1}B\overline{B}^\top(\overline{A}^\top+\gamma_kI)^{-1}\right]\right\}^{-1}$$ once plugged into \eqref{e1} ensures the exponential asymptotic stability in \eqref{e1}, with the decay rate $-\gamma_1$. Here, $0<\gamma_1<\gamma_2<\dots<\gamma_N$ are $N$ positive numbers, large enough such that $A+\gamma_kI$ is invertible for each $k$.
	\end{theorem}

Firstly, let us note that, from the computational point of view, the matrix $K$ given by Theorem \ref{t} requires substantially less resources than the previous one, $K=\frac{1}{2}\overline{B}^\top W^{-1}$. Moreover, there are no additional hypotheses imposed neither on $A$ nor on $B$. We recall that, in order to ease the solving of the Lyapunov equation, usually additional hypotheses are imposed, e.g., $A$ and $B$ are sparse or structured; or,  $M=1$ which allows the use of Ackermann's pole placement method, see \cite{7}.

 In order to prove the Theorem \ref{t}, the main ingredient is to show that the matrix
$$C:=\sum_{k=1}^N\left[(A+\gamma_kI)^{-1}B\overline{B}^\top(\overline{A}^\top+\gamma_kI)^{-1}\right]$$ is invertible, provided that the system \eqref{e1} is controllable. This is done in Lemma \ref{l} below. 

Once proved that $C$ is invertible,  we have
\begin{equation}\label{e2}\begin{aligned}&A+BK\\&
=A-B\overline{B}^\top\sum_{k=1}^N\left[(\overline{A}^\top+\gamma_kI)^{-1}\right]C^{-1}\\&=A-\sum_{k=1}^N\left[(A+\gamma_kI)(A+\gamma_kI)^{-1}B\overline{B}^\top(\overline{A}^\top+\gamma_kI)^{-1}\right]C^{-1}\\&
=A-\left[AC+\sum_{k=1}^N\gamma_kB_k\right]C^{-1}\\&
=-\sum_{k=1}^N\gamma_kB_kC^{-1},\end{aligned}\end{equation}where $B_k:=(A+\gamma_kI)^{-1}B\overline{B}^\top(\overline{A}^\top+\gamma_kI)^{-1},\ k=1,2,.\dots,N.$ 

By the definition, $C$ is a positive semi-definite matrix, and since it is invertible, it is positive definite. Thus, $C^{-1}$ is positive definite as-well. Also, by the definition, $B_k$ is positive semi-definite, $k=1,2,\dots,N$.   Invoking \eqref{e2}, the system \eqref{e1} with $u=Kx,$  $K$ given by Theorem \ref{t}, is equivalent to
$$\dot{x}(t)=\left[A+BK\right]x(t)=-\sum_{k=1}^N\gamma_kB_kC^{-1}x(t).$$ Scalarly multiplying the above equation by $C^{-1}x(t)$, it yields
$$\begin{aligned}&\frac{1}{2}\frac{d}{dt}\left<C^{-1}x(t),x(t)\right>_N=\left<-\sum_{k=1}^N\gamma_kB_kC^{-1}x(t),C^{-1}x(t)\right>_N\\&
=-\left<\gamma_1\sum_{k=1}^NB_kC^{-1}x(t),C^{-1}x(t)\right>_N\\&
\ +\sum_{k=2}^N(\gamma_1-\gamma_k)\left<B_kC^{-1}x(t),C^{-1}x(t)\right>_N\\&
\leq -\gamma_1\left<C^{-1}x(t),x(t)\right>_N,\ t\geq0,\end{aligned}$$since $\sum_{k=1}^NB_k=C$ and $B_k$ is positive semi-definite and $\gamma_k>\gamma_1,\ k=2,3,...,N.$ This implies that 
$$\left<C^{-1}x(t),x(t)\right>_N\leq e^{-2\gamma_1 t}\left<C^{-1}x_0,x_0\right>_N,\ t\geq0,$$ where taking advantage of the fact that $C^{-1}$ is positive-definite, we immediately conclude that 
$$\|x(t)\|_N\leq \mathcal{C}e^{-\gamma_1t}\|x_0\|_N,\ \forall t\geq0,$$for some constant $\mathcal{C}>0$.

Therefore, in order to prove Theorem \ref{t}, it suffices to show that the matrix $C$ is invertible. More precisely, we show that
\begin{lemma}\label{l}Assume that the system \eqref{e1} is controllable, then, the matrix
	$$\sum_{k=1}^N\left[(A+\gamma_kI)^{-1}B\overline{B}^\top(\overline{A}^\top+\gamma_kI)^{-1}\right]$$ is invertible. Here, $0<\gamma_1<\gamma_2<\dots<\gamma_N$ are $N$ positive numbers, large enough such that $A+\gamma_kI$ is invertible for each $k$.
	\end{lemma}
	
	\begin{proof}
		
Recall that the controllability of the system \eqref{e1} implies that the matrix
$$[B\ |\ AB\ |\ A^2B\ | \dots\ |A^{N-1}B] $$ has full rank.

	Let us assume by contradiction that there exists $z\in\mathbb{C}^N,\ z\neq 0,$ such that 
	$$\sum_{k=1}^N(A+\gamma_kI)^{-1}B\overline{B}^\top(\overline{A}^\top+\gamma_kI)^{-1}z=0.$$ That is, the matrix $C$ is not invertible.  Scalarly multiplying in $\mathbb{C}^N$ the above relation by $z$, it yields that
	$$\sum_{k=1}^N \|\overline{B^\top(A^\top+\gamma_kI)^{-1}}z\|_M^2=0$$or, equivalently,
	$$ \overline{B^\top(A^\top+\gamma_kI)^{-1}}z=0,\text{ for all }k=1,2,\dots,N.$$ Or, by taking the conjugate transpose
	$$\overline{z}^\top(A+\gamma_kI)^{-1}B=0,\text{ for all }k=1,2.,\dots,N.$$This implies that there exists a non-trivial linear combination of the lines of the matrix 
	$$\left[(A+\gamma_1I)^{-1}B\ |  \  (A+\gamma_2I)^{-1}B\ | \ \dots\ | \ (A+\gamma_NI)^{-1}B\right]$$ which is zero. Hence, the matrix 
	$$\left[(A+\gamma_1I)^{-1}B\ |  \  (A+\gamma_2I)^{-1}B\ | \ \dots\ | \ (A+\gamma_NI)^{-1}B\right]$$does not have full rank.
	
	It is well-known that  if we multiply a matrix by an invertible square matrix, either on the left or on the right, the resulting matrix has the same rank. That is, the elementary row operations and the corresponding elementary column operations on a matrix preserve the rank of a matrix. Next, we shall perform such elementary operations to show that
	$$\left[(A+\gamma_1I)^{-1}B\ |  \  (A+\gamma_2I)^{-1}B\ | \  \dots\ | \ (A+\gamma_NI)^{-1}B\right]$$ and $$[B\ | \  AB\ | \  A^2B\ | \ \dots\ | \ A^{N-1}B]$$ have the same rank, which clearly is in contradiction with the controllability assumption of the system \eqref{e1}.
	
	 For the ease of presentation, let us assume that $N=3$. For higher $N$, one can argue in a similar fashion. Firstly, let us notice that
	 $$(A+\gamma_iI)^{-1}B-(A+\gamma_jI)^{-1}B=(\gamma_j-\gamma_i)(A+\gamma_iI)^{-1}(A+\gamma_jI)^{-1}B,\ \forall i,j.$$ Multiplying from the right the matrix 
	 $$\left[(A+\gamma_1I)^{-1}B\ |  \  (A+\gamma_2I)^{-1}B\ | \   (A+\gamma_3I)^{-1}B\right] $$ by the invertible matrix 
	 $$\left(\begin{array}{ccc}I&-I&-I\\
	 O&I&O\\
	 O&O&I\end{array}\right),$$here $O$ is the null matrix, then multiplying the result from the left by the invertible matrix $(A+\gamma_1I)$,  we get that
	 \begin{equation}\label{f1}\begin{aligned}&\text{rank}\left[(A+\gamma_1I)^{-1}B\ |  \  (A+\gamma_2I)^{-1}B\ | \   (A+\gamma_3I)^{-1}B\right]\\&
	 =\text{rank}\left[B\ |  \  (\gamma_1-\gamma_2)(A+\gamma_2I)^{-1}B\ | \ (\gamma_1-\gamma_3)(A+\gamma_3I)^{-1}B\right].\end{aligned}\end{equation} Next, multiplying from the right the matrix
	 $$\left[B\ |  \  (\gamma_1-\gamma_2)(A+\gamma_2I)^{-1}B\ | \ (\gamma_1-\gamma_3)(A+\gamma_3I)^{-1}B\right]$$ by the invertible matrix
	 $$\left(\begin{array}{ccc}I&O&O\\
	 O&I&-\frac{\gamma_1-\gamma_3}{\gamma_1-\gamma_2}I\\O&O&I\end{array}\right),$$ then multiplying the result from the left by the invertible matrix $(A+\gamma_2I)(A+\gamma_3I)$,  we deduce that
	 \begin{equation}\label{f2}\begin{aligned}&\text{rank}\left[B\ |  \  (\gamma_1-\gamma_2)(A+\gamma_2I)^{-1}B\ | \ (\gamma_1-\gamma_3)(A+\gamma_3I)^{-1}B\right]\\&
	 =\text{rank}\left[(A+\gamma_2I)(A+\gamma_3I)B\ |  \  (\gamma_1-\gamma_2)(A+\gamma_3I)B\ | \ (\gamma_1-\gamma_3)(\gamma_2-\gamma_3)B\right].\end{aligned}\end{equation} Next, multiplying the matrix
	 $$\left[(A+\gamma_2I)(A+\gamma_3I)B\ |  \  (\gamma_1-\gamma_2)(A+\gamma_3I)B\ | \ (\gamma_1-\gamma_3)(\gamma_2-\gamma_3)B\right]$$ from the right by
	 $$\left(\begin{array}{ccc}I&O&O\\O&\frac{1}{\gamma_1-\gamma_2}I&O\\-\frac{\gamma_3\gamma_2}{(\gamma_1-\gamma_3)(\gamma_2-\gamma_3)}&-\frac{\gamma_3}{(\gamma_1-\gamma_3)(\gamma_2-\gamma_3)}I&\frac{1}{(\gamma_1-\gamma_3)(\gamma_2-\gamma_3)}I\end{array}\right)$$we get that
	 \begin{equation}\label{f3}\begin{aligned}&\text{rank}\left[(A+\gamma_2I)(A+\gamma_3I)B\ |  \  (\gamma_1-\gamma_2)(A+\gamma_3I)B\ | \ (\gamma_1-\gamma_3)(\gamma_2-\gamma_3)B\right]\\&
	 =\text{rank}\left[(A^2+(\gamma_2+\gamma_3)A)B\ |  \  (\gamma_1-\gamma_2)AB\ | \ B\right]\end{aligned}.\end{equation} Finally, multiplying from the right the matrix
	 $$\left[(A^2+(\gamma_2+\gamma_3)A)B\ |  \  (\gamma_1-\gamma_2)AB\ | \ B\right]$$ by
	 $$\left(\begin{array}{ccc}I&O&O\\
	 -\frac{\gamma_2+\gamma_3}{\gamma_1-\gamma_2}I&\frac{1}{\gamma_1-\gamma_2}I&O\\
	 O&O&I\end{array}\right)$$ we get that
	 \begin{equation}\label{f4}\begin{aligned}&\text{rank}\left[(A^2+(\gamma_2+\gamma_3)A)B\ |  \  (\gamma_1-\gamma_2)AB\ | \ B\right]\\&
	 =\text{rank}[A^2B\ | \ AB \ | \ B]\\&
	 =\text{rank}[B\ | \ AB  \ | \ A^2B].\end{aligned}\end{equation}The conclusion follows immediately from \eqref{f1}-\eqref{f4}. Thereby, completing the proof.
\end{proof}

\section{Conclusions} The present work provides a feedback form controller for the stabilization of finite dimensional differential systems. Previous feedback forms, from the existing literature, are under the requirement of a huge effort from the computational point of view. In the present case, we improved a lot this aspect. The result has impact on the design of feedback controllers for the stabilization of finite-dimensional non-linear systems, via the first approximation method see \cite{1}; or for  the stabilization of infinite-dimensional systems via the spectral decomposition method, see starting with the pioneering work \cite{3} up to the highly complex case of the Navier-Stokes equations in \cite{4}. This result is also related to the design of proportional type feedback stabilizing controllers in \cite{5}.

\end{document}